\documentclass[12pt]{article}
\usepackage{amsmath}
\usepackage[english]{babel}
\usepackage{eucal}
\usepackage{stmaryrd}
\usepackage{mathabx}
\usepackage{latexsym}
\usepackage{amscd}
\usepackage{amsthm}

\newtheorem{theorem}{Theorem}
\newtheorem{lemma}[theorem]{Lemma}

\theoremstyle{definition}
\newtheorem{definition}[theorem]{Definition}
\newtheorem{algorithm}[theorem]{Algorithm}

\newtheorem{example}[theorem]{Example}

\begin{document}

\begin{center}
\textbf{\large{Markovsky algorithm on $i$-invertibile groupoids}}
\end{center}

\begin{center}
\textsc{Nadeghda N. Malyutina and Alexandra V. Scerbacova  and  Victor A. Shcherbacov}
\end{center}

\bigskip

\textsc{Abstract.}  We propose modification of Markovsky crypto-algorithm \cite{MARKOVSKI, VS_CRYPTO_12, 2017_Scerb}   on $n$-ary groupoids \cite{2}  that are  invertible on at least one place.

\medskip

\textbf{Keywords:} left quasigroup, $n$-ary quasigroup, $i$-invertibile groupoid, Markovsky algorithm

\medskip

\textbf{AMS:} 20N15, 20N05,  05B15, 94A60

\medskip

We continue researches  of applications of $n$-ary groupoids that are invertible on $i$-th place  in cryptology \cite{2018_Malutina}.

It is clear that Markovsky crypto-algorithm which is based on  binary or $n$-ary quasigroup has  better \lq\lq mixing properties\rq\rq \, than proposed algorithm.

But from the other side it is well known \cite{VOEVODE, VOIVODA_04}  that binary ($n$-ary) quasigroup used  in  Markovsky algorithm  is its  key. It is clear that number of $i$-invertible $n$-groupoids (number $n$ is fixed) is more then number of $n$-ary quasigroups (number $n$ is fixed).

\medskip

\begin{definition}
n-Ary groupoid $(Q, f)$ is called invertible on  the $i$-th place, $i\in \overline{1,n}$, if the equation $f(a_1, \dots, a_{i-1}, x_i, a_{i+1}, \dots, a_n)=a_{n+1}$  has unique solution for any elements:
 $a_1,  \dots,   a_{i-1}, $ $  a_{i+1}, \dots, $ $a_n, a_{n+1} \in Q$ \cite{2}.

\end{definition}

In this case operation  ${}^{(i, n+1)}f(a_1, \dots, a_{i-1}, a_{n+1}, a_{i+1}, \dots, a_n)= x_i$ is defined in unique way and we have:

\begin{equation} \label{seven_one}
\begin{split}
 f(a_1, \dots, a_{i-1}, {}^{(i, n+1)}f(a_1, \dots, a_{i-1}, a_{n+1}, a_{i+1}, \dots, a_n),  a_{i+1}, \dots, a_n)
= a_{n+1}, \\
  {}^{(i, n+1)}f(a_1, \dots, a_{i-1}, f(a_1, \dots, a_{i-1}, x_i, a_{i+1}, \dots, a_n),   a_{i+1}, \dots, a_n)
= x_i.
\end{split}
\end{equation}

A translation of $i$-invertible $n$-ary groupoid  $(Q, f)$ ($n>2$) will be denoted as $ T(a_1, \dots, $ $a_{i-1}, -,$
$a_{i+1}, \dots, a_n)$, where $a_i\in Q$ for all $i\in \overline {1,n}$ and
\begin{equation*}
T(a_1, \dots, a_{i-1}, -, a_{i+1}, \dots, a_n) x = f(a_1, \dots, a_{i-1}, x, a_{i+1}, \dots, a_n)
\end{equation*}
for all $x\in Q$.

 From the definition of $i$-invertible $n$-ary groupoid  $(Q,f)$ it follows that any translation of the groupoid
 $(Q,f)$ is a permutation of the set $Q$.
 In the  next lemma we suppose that $i=n$. It is clear that next lemma is true for any other value of variable $i$.

\begin{lemma} \label{INVERSE_TRANS_IN_N_PAR_i}
If ${}_fT(a_1, \dots, a_{n-1}, -)$ is a translation of an $i$-invertible $n$-groupoid  $(Q, f)$, then
$${}_fT^{-1}(a_1, \dots, a_{n-1}, -) = {}_{{}^{(n, n+1)}f} T(a_1, \dots, a_{n-1}, -). $$
\end{lemma}
\begin{proof}
In the proof we omit the symbol $f$ in the notation of translations of the groupoid  $(Q, f)$.
We have
\begin{equation}
\begin{split}
& T^{-1} (a_1, \dots, a_{n-1}, -) (T(a_1, \dots, a_{n-1}, -)x) = \\ &
T^{-1} (a_1, \dots, a_{n-1}, -) f(a_1, \dots,  a_{n-1}, x) = \\ &
{{}^{(n, n+1)}f} (a_1, \dots, a_{n-1}, f(a_1, \dots,  a_{n-1}, x)) \overset{(\ref{seven_one})}{=}  x.
\end{split}
\end{equation}
\end{proof}

\begin{algorithm} \label{ALGn_ar} Let $Q$ be a non-empty finite alphabet and  $k$ be a natural number, $u_i,
v_i \in Q$, $i\in \{1,..., k\}$.  Define an $n$-ary groupoid  $(Q, f)$ which is invertible on $n$-th  place.
It is clear that  groupoid  $(Q, {}^{(n,\, n+1)}f)$ is defined in a unique way.

Take the fixed elements $l_1^{(n^2-n)/2}$ ($l_i\in Q$), which  are called  leaders.

Let $u_1 u_2... u_k$ be a $k$-tuple  of letters from $Q$, $a, b, c, d, \dots$ are natural numbers.

  \begin{equation} \label{ALG_4_equations_N_ARY_TR_i}
\begin{split}
& v_1 = T^a(l_1, l_2, \dots, l_{n-1}, u_1),  \\
& v_2 = T^b(l_n, l_{n+1}, \dots, l_{2n-3}, v_1,  u_2), \\
& \dots , \\
& v_{n-1} = T^c(l_{(n^2 - n)/2}, v_1,\dots,  v_{(n-2)}, u_{n-1}), \\
& v_{n}= T^d(v_{1}, \dots, v_{n-1}, u_{n}), \\
& v_{n+1}= T^e(v_{2}, \dots, v_{n}, u_{n+1}), \\
& v_{n+2}= T^t(v_{3}, \dots, v_{n+1}, u_{n+2}), \\
& \dots
\end{split}
\end{equation}
Therefore we obtain the following ciphertext $v_1v_2 \dots v_k$.

\medskip

 Taking in consideration Lemma \ref{INVERSE_TRANS_IN_N_PAR_i} we can say that the  deciphering algorithm can be constructed similar to  the deciphering  Algorithm given in \cite{2017_Scerb, 2018_Malutina}.

\end{algorithm}

\medskip

\begin{example}
We construct ternary groupoid $(R_3, f)$, $R_3 =\{0, 1, 2 \}$, which is defined over the ring $(R_3, +, \cdot)$ of residues modulo 3 and which is invertible on the third place. We define ternary operation $f$ on the set $R_3$ in the following way:
$f(x_1, x_2, x_3) = \alpha x_1 + \beta x_2 +  x_3 = x_4$, where $\alpha 0= 2,  \alpha 1=2,  \alpha 2=0, $
$\beta 0 = 1$, $\beta 1 = 1$, $\beta 2 = 1$.

Below $T_{2,0} 1 = 2$ means that $f(2, 0, 1) = 2$ and so on.

We have:

$T_{0, 0} 0 = 0$, $ T_{0,0} 1 = 1, $ $T_{0,0} 2 = 2$, $T_{0,1} 0 = 0$,  $T_{0,1} 1 = 1,$ $T_{0,1} 2 = 2$,

$T_{0,2} 0 = 0$,  $T_{0,2} 1 = 1$,  $T_{0,2} 2 = 2$, $T_{1,0} 0 = 0$,  $T_{1,0} 1 = 1$,  $T_{1,0} 2 = 2$,

$T_{1,1} 0 = 0$,  $T_{1,1} 1 = 1$,  $T_{1,1} 2 = 2$, $T_{1,2} 0 = 0$,  $T_{1,2} 1 = 1$,  $T_{1,2} 2 = 2$,

$T_{2,0} 0 = 1$,  $T_{2,0} 1 = 2$,  $T_{2,0} 2 = 0$, $T_{2,1} 0 = 1$,  $T_{2,1} 1 = 2$,  $T_{2,1} 2 = 0$,

$T_{2,2} 0 = 1$,  $T_{2,2} 1 = 2$,  $T_{2,2} 2 = 0$.

\bigskip

In this case ${}^{(3.4)}f(x_1, x_2, x_4)   = x_3 = 2\cdot \alpha x_1 + 2\cdot \beta x_2 + x_4. $

\bigskip

Check. $f(x_1, x_2, x_3) = f(x_1, x_2, {}^{(3.4)}f(x_1, x_2, x_4)) = \alpha x_1 + \beta x_2  + 2\cdot \alpha x_1 + 2\cdot \beta x_2 + x_4 = x_4$.

${}^{(3.4)}f(x_1, x_2, f(x_1, x_2, x_3))   = 2\cdot \alpha x_1 + 2\cdot \beta x_2 + \alpha x_1 + \beta x_2 +  x_3 = x_3 $. We propose the following elements  $l_1 = 2, l_2 = 0, l_3 = 2$ as leader elements. In Algorithm \ref{ALG_4_equations_N_ARY_TR_i} we put $a=1$, $b=2$, $c=1$, $d=2$ and so on.

In this case open text $2\, 0\, 1\, 1\, 2\, 1$ is transformed in the following crypto-text:

$T^1_{l_1, l_2} u_1 = f(l_1, l_2, u_1) = f(2, 0, 2) = 0 = v_1$,

$T^2_{l_3, v_1} u_2 = f(2, 0, f (2, 0, 0)) = f(2, 0, 1) = 2 = v_2$,

$T^1_{v_1, v_2} u_3 = f(0, 2, 1) =  1 = v_3$,

$T^2_{v_2, v_3} u_4 = f(2, 1, f (2, 1, 1)) = f(2, 1, 2) = 0 = v_4$,

$T^1_{v_3, v_4} u_5 = f(1, 0, 2) =  2 = v_5$,

$T^2_{v_4, v_5} u_6 = f(0, 2, f (0, 2, 1)) = f(0, 2, 1) = 1 = v_6$.

\medskip

We obtain the following crypto-text $0\, 2\, 1\, 0\, 2\, 1$.

We have the following deciphering procedure. Notice that in conditions of this example  the following fact is true: $T^{-1} (x, y, -) = T^{2} (x, y, -)$.

\smallskip

$T^2_{l_1, l_2} v_1 = f(l_1, l_2, f(l_1, l_2, v_1))  =  f(2, 0, f(2, 0, 0)) = f(2, 0, 1) = 2 =  u_1$,

$T^1_{l_3, v_1} v_2 = f(2, 0, 2)  =  0 = u_2$,

$T^2_{v_1, v_2} v_3 = f(0, 2, f(0, 2, 1)) = f(0, 2, 1) =  1 = u_3$,

$T^1_{v_2, v_3} v_4 = f(v_2, v_3, v_4) = f(2, 1, 0) = 1 = u_4$,

$T^2_{v_3, v_4} v_5 = f(v_3, v_4, f(v_3, v_4, v_5)) =  f(1, 0, f(1, 0, 2)) =
f(1, 0,  2) = 2 =   u_5$,

$T^1_{v_4, v_5} v_6 = f(0, 2, 1) = 1 = u_6$.

\smallskip

Therefore we have the following open text: $2\, 0\, 1\, 1\, 2\, 1$.
\end{example}

\bigskip

\vspace{2mm}
\begin{parbox}{118mm}{\footnotesize Nadeghda  Malyutina$^{1}$, Alexandra Scerbacova$^{2}$,  Victor Shcherbacov$^{3}$
\vspace{3mm}

\noindent
$^{1}$Senior Lecturer//Shevchenko Transnistria State University

\noindent
Email: 231003.bab.nadezhda@mail

\vspace{3mm}

\noindent
$^{2}$Master Student//Gubkin Russian State Oil and Gas University

\noindent
Email: scerbik33@yandex.ru

\vspace{3mm}

\noindent
$^{3}$Principal Researcher//
Institute of Mathematics and Computer Science

\noindent Email:
 victor.scerbacov@math.md
}
\end{parbox}

\end{document}